\documentclass[11pt]{article}

\title{Order dimension beyond rank for simplicial hyperplane arrangements}
\author{Daria Poliakova}
\date{}

\usepackage[a4paper,margin=2.6cm]{geometry}
\usepackage{amsmath,amssymb,amsthm}
\usepackage{mathtools}
\usepackage{graphicx}
\usepackage{enumitem}
\usepackage[hidelinks]{hyperref}

\usepackage{tikz}
\usetikzlibrary{arrows.meta,calc,positioning}

\setlist[itemize]{leftmargin=2em}
\setlist[enumerate]{leftmargin=2.2em}

\usepackage{accents}
\newlength{\dhatheight}

\newtheorem{theorem}{Theorem}[section]

\theoremstyle{definition}

\begin{document}

\maketitle

\begin{abstract}
We show that the order dimension of the poset of regions in a simplicial hyperplane arrangement can exceed its rank, answering a question of Reading and Segovia. Examples are Coxeter arrangements \(H_4\) and \(E_6\), with \( \dim W(H_4) \geq 5\) and \( \dim W(E_6) \geq 7\). 
\end{abstract}

\section*{Introduction}

We show that the order dimension of the poset of regions of a simplicial hyperplane arrangement can exceed its rank: the weak orders of types \(H_4\) and \(E_6\) satisfy \(\dim W(H_4)\geq 5>4\) and \(\dim W(E_6)\geq 7>6\). The second example shows that this phenomenon occurs already for crystallographic arrangements.

Let \(\mathcal A\) be a finite central essential hyperplane arrangement in \(\mathbb R^n\), and let \(B\) be a region. For a region \(R\), let \(S_B(R)\) be the set of hyperplanes separating \(R\) from \(B\). The poset of regions \(P(\mathcal A,B)\) is defined by \(R\leq R'\) if \(S_B(R)\subseteq S_B(R')\). Its Hasse diagram is the adjacency graph of the regions, oriented away from \(B\). Equivalently, \(P(\mathcal A,B)\) is the vertex poset obtained from a suitable linear-functional orientation of the \(1\)-skeleton of the zonotope associated with \(\mathcal A\) \cite{Reading2003}. If \(\mathcal A\) is simplicial, this zonotope is simple and \(P(\mathcal A,B)\) is a lattice for every choice of \(B\) \cite{BjornerEdelmanZiegler1990}.

A realizer of a finite poset \(P\) is a collection of linear extensions whose intersection is \(P\), and the order dimension \(\dim P\) is the minimum size of a realizer. Equivalently, it is the least \(d\) such that \(P\) embeds into \(\mathbb R^d\) with the coordinatewise order \cite{DushnikMiller1941}. If \(B\) has \(k\) walls, the corresponding atoms and their antipodal coatoms form the standard example \(S_k\), so \(\dim P(\mathcal A,B)\geq k\). In particular, \(\dim P(\mathcal A,B)\geq n\) when \(\mathcal A\) is simplicial of rank \(n\) \cite[Proposition~3.1]{Reading2003}.

For the reflection arrangement of a finite Coxeter group, the poset of regions is its weak order. Reading proved equality between dimension and rank in every irreducible finite type except \(E_6,E_7,E_8,F_4\), and \(H_4\), for which he obtained bounds. He observed that an exceptional type whose dimension exceeds its rank would give the first known simplicial arrangement with this property \cite[Section~8]{Reading2003}, and later recorded the general guess that every simplicial region poset has dimension equal to its rank \cite[Problem~9.3]{Reading2016}. More recently, Segovia asked the analogous question for lattices arising from linear-functional orientations of polytopal \(1\)-skeleta \cite[Question~5.5]{Segovia2025}.

We answer this question negatively by proving \(\dim W(H_4)\geq 5\) and \(\dim W(E_6)\geq 7\). For ordered incomparable pairs \(p_i=(x_i,y_i)\), join \(p_i\) and \(p_j\) when \(x_i\leq y_j\) and \(x_j\leq y_i\). No linear extension can reverse two adjacent pairs, so every realizer induces a proper coloring of this incompatibility graph. Consequently, its chromatic number is a lower bound for \(\dim P\); this is the usual alternating-\(2\)-cycle argument \cite{FelsnerTrotter2000}. For each of \(H_4\) and \(E_6\), we exhibit fifteen incomparable pairs whose incompatibility graph has chromatic number at least \(5\) and \(7\), respectively. All necessary weak-order comparisons and both short noncolorability arguments are included, so the counterexamples can be checked by hand (even though computations of the inversion sets are unpleasant and the author honestly does not recommend them).

Since simplicial region lattices are semidistributive \cite{ReadingLattice2003}, our examples also answer the unrestricted and semidistributive versions of Segovia's question negatively, even for simple zonotopes. They do not address the extremal version.

Computer calculations show in fact that \(\dim W(H_4)=5 \), \(\dim W(E_6)=8\), \(10\leq\dim W(E_7)\leq11 \), and additionally \(\dim W(F_4)=4 \). The corresponding obstruction subgraphs and realizers are large, and since no conceptual pattern emerges from them, we have chosen not to include them in the present note. We would, however, be happy to share the code upon request.

Our results raise two natural questions.

{\bf Question 1.} Which geometric, combinatorial, or lattice-theoretic conditions on \((\mathcal A,B)\) guarantee \(\dim P(\mathcal A,B)=\operatorname{rk}(\mathcal A)\)?

{\bf Question 2.} Let \(f(n)\) be the supremum of \(\dim P(\mathcal A,B)\) over finite central simplicial arrangements of rank \(n\) and all choices of \(B\). Is \(f(n)\) finite, and, if so, what is its asymptotic growth? 

\paragraph{AI use declaration.} The small obstruction subgraphs were found by ChatGPT 5.6 Sol Ultra. The human input was the belief that the rank guess is incorrect, and one should look for counterexamples.

\paragraph{Acknowledgements and funding.} I am grateful to Andrii Bondarenko for encouraging me to discuss this question with ChatGPT. I was funded by Deutsche Forschungsgemeinschaft via SFB 1624. I am grateful to BIMSA, where the note was finalized, for hospitality. After making the note public I learned that Adrien Segovia has independently arrived at similar obstruction for \(H_4\); I am grateful to him for sharing his finding.

\section{Preliminaries}
\subsection{Coxeter arrangements and weak order}
\label{subsec:coxeter}

Let \((W,S)\) be a finite Coxeter system of rank \(n\), with \(S=\{s_1,\ldots,s_n\}\), simple roots \(\Delta=\{\alpha_1,\ldots,\alpha_n\}\), and positive roots \(\Phi^+\). Its Coxeter arrangement is \(\mathcal A_W=\{\alpha^\perp:\alpha\in\Phi^+\}\). Let \(B\) be the fundamental region, whose interior is \(\{v:\langle v,\alpha_i\rangle>0\text{ for all }i\}\). The regions of \(\mathcal A_W\) are the \(wB\), for \(w\in W\).

Set \(\Phi^-=-\Phi^+\) and define \(\operatorname{Inv}(w)=\{\alpha\in\Phi^+:w^{-1}(\alpha)\in\Phi^-\}\). The hyperplanes separating \(wB\) from \(B\) are precisely those indexed by \(\operatorname{Inv}(w)\). Thus \(wB\mapsto w\) identifies \(P(\mathcal A_W,B)\) with the right weak order: \(u\leq_R v\) if \(v=uq\) and \(\ell(v)=\ell(u)+\ell(q)\), where \(\ell\) denotes Coxeter length, or equivalently if \(\operatorname{Inv}(u)\subseteq\operatorname{Inv}(v)\). We write \(W(X)\) for the weak order of type \(X\) \cite[Section~5]{Reading2003}\cite[Chapter~3]{BjornerBrenti2005}.

Let \(m_{ij}\) be the order of \(s_is_j\). In type \(H_4\), \(m_{12}=5\) and \(m_{23}=m_{34}=3\); in type \(E_6\), \(m_{13}=m_{34}=m_{24}=m_{45}=m_{56}=3\). All unlisted \(m_{ij}\) equal \(2\). We abbreviate \(s_{i_1}\cdots s_{i_k}\) to \(i_1\cdots i_k\). If \(\ell(ws_i)=\ell(w)+1\), then \(\operatorname{Inv}(ws_i)=\operatorname{Inv}(w)\cup\{w(\alpha_i)\}\), which computes inversion sets recursively from reduced words \cite[Proposition~2.1]{HohlwegLabbe2015}. We write roots in simple-root coordinates: \([c_1,\ldots,c_n]=\sum_i c_i\alpha_i\).

Let \(w_0\) be the longest element and put \(\omega=-w_0\), regarded as a linear map on the root space. Then \(\omega\) permutes \(\Phi^+\) and \(\operatorname{Inv}(w_0z)=\Phi^+\setminus\omega(\operatorname{Inv}(z))\). Consequently,
\begin{equation}
 x\leq w_0z
 \quad\Longleftrightarrow\quad
 \operatorname{Inv}(x)\cap\omega(\operatorname{Inv}(z))=\varnothing.
 \label{eq:comparison}
\end{equation}
Indeed, for \(\alpha\in\Phi^+\), one has \((w_0z)^{-1}\alpha=-z^{-1}\omega(\alpha)\), so \(\alpha\in\operatorname{Inv}(w_0z)\) precisely when \(\omega(\alpha)\notin\operatorname{Inv}(z)\). In type \(H_4\), \(\omega\) is the identity. In type \(E_6\), it interchanges \(\alpha_1\leftrightarrow\alpha_6\) and \(\alpha_3\leftrightarrow\alpha_5\), while fixing \(\alpha_2\) and \(\alpha_4\) \cite[Appendix~A1]{BjornerBrenti2005}. Thus all comparisons below reduce to disjointness checks between finite sets of positive roots.

\subsection{Incomparable pairs and graph coloring}
\label{subsec:critical-pairs}

For a finite poset \(P\), let \(\operatorname{Inc}(P)=\{(x,y):x\parallel y\}\), where pairs are ordered. A linear extension \(L\) reverses \((x,y)\) if \(y<_Lx\), and \(I\subseteq\operatorname{Inc}(P)\) is reversible if one linear extension reverses every pair in \(I\). An alternating cycle is a cyclic sequence \((x_i,y_i)_{i=1}^k\), with \(k\geq2\), such that \(x_i\leq y_{i+1}\), where indices are taken modulo \(k\). A set is reversible if and only if it contains no alternating cycle, and \(\dim P\) is the least number of such sets covering \(\operatorname{Inc}(P)\) \cite[Theorem~2.3]{FelsnerTrotter2000}.

For \(I\subseteq\operatorname{Inc}(P)\), define the incompatibility graph \(\Gamma_P(I)\) on \(I\) by joining \((x_i,y_i)\) and \((x_j,y_j)\) when \(x_i\leq y_j\) and \(x_j\leq y_i\). Adjacent pairs form an alternating \(2\)-cycle and cannot be reversed by the same linear extension. Hence
\begin{equation}
 \chi(\Gamma_P(I))\leq\dim P.
 \label{eq:incompatibility-bound}
\end{equation}
The converse need not hold even for \(I = \operatorname{Inc}(P)\): an independent set may contain an alternating cycle of length at least \(3\). Thus graph noncolorability certifies lower bounds, whereas upper bounds require reversible classes or actual linear extensions \cite[Lemma~3.3 and Section~4]{FelsnerTrotter2000}.

The computer search leading to the obstruction subgraphs presented below was in fact not over full \(\operatorname{Inc}(P) \) but over a much smaller subset \( \operatorname{Crit}(P) \) of critical pairs, which does not change the dimension; the proofs below, however, use only that the displayed pairs are incomparable \cite[Proposition~3.2]{FelsnerTrotter2000}, \cite[Proposition~3.3]{Reading2003}.

\section{The \(H_4\) lower bound}
\label{sec:h4-lower}

\begin{theorem}
\label{thm:h4-lower}
\(\dim W(H_4)\geq5\).
\end{theorem}

\begin{proof}
The root system of type \(H_4\) has \(60\) positive roots, and hence
\(\ell(w_0)=60\)
\cite[Theorem~7.1.5 and Appendix~A1]{BjornerBrenti2005}.
Put \(\varphi=(1+\sqrt5)/2\), and normalize the simple roots by
\(\langle\alpha_i,\alpha_i\rangle=2\),
\(\langle\alpha_1,\alpha_2\rangle=-\varphi\), and
\(\langle\alpha_2,\alpha_3\rangle
=\langle\alpha_3,\alpha_4\rangle=-1\), with all other off-diagonal
inner products zero. The
opposition map \(\omega=-w_0\) is the identity in type \(H_4\).

For the words in Table~\ref{tab:h4-pairs}, put
\(p_i=(x_i,w_0z_i)\). By performing the tedious yet straightforward
inversion-set recursion of Section~\ref{subsec:coxeter}, the reader can
verify that all displayed words are reduced and that \( \operatorname{Inv}(x_i)\cap\operatorname{Inv}(z_i)=\{\beta_i\} \).

\begin{table}[ht]
\centering
\scriptsize
\setlength{\tabcolsep}{2pt}
\begin{tabular}{rlll}
\hline
\(i\) & \(x_i\) & \(z_i\) & \(\beta_i\)\\
\hline
1  & \(1\)                    & \(1\)                                      & \([1,0,0,0]\)\\
2  & \(2\)                    & \(2\)                                      & \([0,1,0,0]\)\\
3  & \(3\)                    & \(3\)                                      & \([0,0,1,0]\)\\
4  & \(4\)                    & \(4\)                                      & \([0,0,0,1]\)\\
5  & \(23\)                   & \(32\)                                     & \([0,1,1,0]\)\\
6  & \(212\)                  & \(121\)                                    & \([\varphi,\varphi,0,0]\)\\
7  & \(32121\)                & \(123\)                                    & \([\varphi,1,1,0]\)\\
8  & \(121234\)               & \(43212132123\)                            & \([\varphi^2,\varphi^2,1,1]\)\\
9  & \(432121\)               & \(1234\)                                   & \([\varphi,1,1,1]\)\\
10 & \(432123\)               & \(21213212432121321\)                      & \([\varphi,\varphi^2,\varphi,\varphi]\)\\
11 & \(21321234\)             & \(43212132123432121321234321213212432121\) & \([\varphi^2,\varphi^3,\varphi^2,1]\)\\
12 & \(3212132143212\)        & \(43212132143212\)                         & \([\varphi^2,2\varphi,2\varphi,\varphi]\)\\
13 & \(121232121321432123\)   & \(3432121321234321213212343212132123\)     & \([2+2\varphi,1+3\varphi,1+2\varphi,\varphi]\)\\
14 & \(321243212132143212\)   & \(212132123432121321234321213212\)         & \([1+2\varphi,1+3\varphi,1+2\varphi,\varphi]\)\\
15 & \(21321213212432121321\) & \(4321213212432121\)                       & \([\varphi^2,\varphi^3,\varphi^2,\varphi]\)\\
\hline
\end{tabular}
\caption{The words defining the pairs \(p_i=(x_i,w_0z_i)\).}
\label{tab:h4-pairs}
\end{table}

Since
\(\operatorname{Inv}(x_i)\cap\operatorname{Inv}(z_i)\neq\varnothing\),
equation~\eqref{eq:comparison} gives \(x_i\nleq w_0z_i\). Inspection
of the displayed reduced words shows that \( \ell(w_0z_i)=\ell(w_0)-\ell(z_i)>\ell(x_i) \), so \(w_0z_i\nleq x_i\). Thus every \(p_i\) is an ordered incomparable
pair.

We now define a graph \(G\) on these pairs. The vertices
\(p_1,p_2,p_3,p_4\) span a \(K_4\). For \(i\geq5\), join \(p_i\) to
\(p_j\) for every \( j\in\bigl(\{1,2,3,4\}\setminus L_i\bigr)\cup N_i \),
where \(L_i\) and \(N_i\) are given in Table 2.

\begin{table}[ht]
\centering
\begin{tabular}{rll@{\qquad}rll}
\hline
\(i\) & \(L_i\) & \(N_i\) & \(i\) & \(L_i\) & \(N_i\)\\
\hline
5  & \(\{2,3\}\)   & \(\{6,13\}\)       & 11 & \(\{2,4\}\)   & \(\{13\}\)\\
6  & \(\{1,2\}\)   & \(\{7,9,11,15\}\) & 12 & \(\{3,4\}\)   & \(\{15\}\)\\
7  & \(\{1,3\}\)   & \(\{14\}\)        & 13 & \(\{1,3,4\}\) & \(\varnothing\)\\
8  & \(\{1,4\}\)   & \(\{12,13,15\}\)  & 14 & \(\{2,3\}\)   & \(\varnothing\)\\
9  & \(\{1,4\}\)   & \(\{10\}\)        & 15 & \(\{2,3,4\}\) & \(\varnothing\)\\
10 & \(\{2,4\}\)   & \(\{14\}\)        &    &                   & \\
\hline
\end{tabular}
\caption{The sets \(L_i\) and \(N_i\) defining \(G\).}
\label{tab:h4-graph}
\end{table}

By comparing the inversion sets obtained from the same recursion, the reader
can verify that, for every edge \(p_ip_j\) of \(G\), \( \operatorname{Inv}(x_i)\cap\operatorname{Inv}(z_j)
 =\operatorname{Inv}(x_j)\cap\operatorname{Inv}(z_i)
 =\varnothing \). 
Equation~\eqref{eq:comparison} therefore gives
\(x_i\leq w_0z_j\) and \(x_j\leq w_0z_i\), so \(p_i,p_j\) form an
alternating \(2\)-cycle. Thus \( G\subseteq
 \Gamma_{W(H_4)}(\{p_1,\ldots,p_{15}\}) \).

Suppose that \(G\) has a \(4\)-coloring, where
\(p_1,p_2,p_3,p_4\) have colors \(A,B,C,D\), respectively. By the
definition of \(L_i\), the color of \(p_i\), for \(i\geq5\), must be
the color of some \(p_j\) with \(j\in L_i\). In particular, \(p_6\)
has color \(A\) or \(B\).

If \(p_6\) has color \(A\), the edges
\(p_6p_7,p_6p_9,p_9p_{10},p_{10}p_{14}\) successively force
\(p_7,p_9,p_{10},p_{14}\) to have colors \(C,D,B,C\), contradicting
the edge \(p_7p_{14}\).

If \(p_6\) has color \(B\), the edges \(p_5p_6\) and \(p_6p_{11}\)
force \(p_5\) and \(p_{11}\) to have colors \(C\) and \(D\). Since
both are adjacent to \(p_{13}\), the vertex \(p_{13}\) has color
\(A\); the edges \(p_{13}p_8\) and \(p_8p_{12}\) then force
\(p_8,p_{12}\) to have colors \(D,C\). Finally, since \(p_{15}\) is
adjacent to both \(p_6\) and \(p_8\), it has color \(C\), contradicting
the edge \(p_{12}p_{15}\).

Hence \(\chi(G)\geq5\). Equation~\eqref{eq:incompatibility-bound}
now gives
\[
 5\leq\chi(G)\leq\dim W(H_4).
\]
\end{proof}

\section{The \(E_6\) lower bound}
\label{sec:e6-lower}

\begin{theorem}
\label{thm:e6-lower}
\(\dim W(E_6)\geq7\).
\end{theorem}

\begin{proof}
The root system of type \(E_6\) has \(36\) positive roots, and hence
\(\ell(w_0)=36\)
\cite[Theorem~7.1.5 and Appendix~A1]{BjornerBrenti2005}.
The opposition map \(\omega=-w_0\) is nontrivial in type \(E_6\). In
simple-root coordinates, \( \omega([c_1,c_2,c_3,c_4,c_5,c_6])
 =[c_6,c_2,c_5,c_4,c_3,c_1] \).

For the words in Table~\ref{tab:e6-pairs}, put \(p_i=(x_i,w_0z_i)\). Just in the previous section, but with the nontrivial opposition map \(\omega \), all displayed words are reduced, and \( \operatorname{Inv}(x_i)\cap\omega(\operatorname{Inv}(z_i))={\beta_i} \), and \( \ell(w_0z_i)>\ell(x_i) \), hence every \(p_i \) is an ordered incomparable pair.

\begin{table}[ht]
\centering
\scriptsize
\setlength{\tabcolsep}{3pt}
\begin{tabular}{rlll}
\hline
\(i\) & \(x_i\) & \(z_i\) & \(\beta_i\)\\
\hline
1  & \(1\)        & \(6\)        & \([1,0,0,0,0,0]\)\\
2  & \(2\)        & \(2\)        & \([0,1,0,0,0,0]\)\\
3  & \(3\)        & \(5\)        & \([0,0,1,0,0,0]\)\\
4  & \(4\)        & \(4\)        & \([0,0,0,1,0,0]\)\\
5  & \(5\)        & \(3\)        & \([0,0,0,0,1,0]\)\\
6  & \(6\)        & \(1\)        & \([0,0,0,0,0,1]\)\\
7  & \(65431\)    & \(65431\)    & \([1,0,1,1,1,1]\)\\
8  & \(2431\)     & \(6542\)     & \([1,1,1,1,0,0]\)\\
9  & \(2456\)     & \(1342\)     & \([0,1,0,1,1,1]\)\\
10 & \(3456\)     & \(1345\)     & \([0,0,1,1,1,1]\)\\
11 & \(3425431\)  & \(65423456\) & \([1,1,2,2,1,0]\)\\
12 & \(6542\)     & \(2431\)     & \([0,1,0,1,1,1]\)\\
13 & \(1342\)     & \(2456\)     & \([1,1,1,1,0,0]\)\\
14 & \(1345\)     & \(3456\)     & \([1,0,1,1,1,0]\)\\
15 & \(65423456\) & \(3425431\)  & \([0,1,1,2,2,1]\)\\
\hline
\end{tabular}
\caption{The words defining the pairs \(p_i=(x_i,w_0z_i)\).}
\label{tab:e6-pairs}
\end{table}

We now define a graph \(G\) on these pairs. The vertices
\(p_1,p_2,p_3,p_4,p_5,p_6\) span a \(K_6\). For \(i\geq7\), join \(p_i\) to
\(p_j\) for every \( j\in\bigl(\{1,2,3,4,5,6\}\setminus L_i\bigr)\cup N_i \),
where \(L_i\) and \(N_i\) are given in Table 4.

\begin{table}[ht]
\centering
\begin{tabular}{rll@{\qquad}rll}
\hline
\(i\) & \(L_i\) & \(N_i\) & \(i\) & \(L_i\) & \(N_i\)\\
\hline
7  & \(\{1,6\}\) & \(\{8,11,12,15\}\) & 12 & \(\{2,6\}\) & \(\{13\}\)\\
8  & \(\{1,2\}\) & \(\{9\}\)          & 13 & \(\{1,2\}\) & \(\{14\}\)\\
9  & \(\{2,6\}\) & \(\{10\}\)         & 14 & \(\{1,5\}\) & \(\{15\}\)\\
10 & \(\{3,6\}\) & \(\{11\}\)         & 15 & \(\{5,6\}\) & \(\varnothing\)\\
11 & \(\{1,3\}\) & \(\varnothing\)     &    &             & \\
\hline
\end{tabular}
\caption{The sets \(L_i\) and \(N_i\) defining \(G\).}
\label{tab:e6-graph}
\end{table}

The existence of these edges in \(
\Gamma_{W(E_6)}(\{p_1,\ldots,p_{15}\}) \) is verified exactly like in the previous section, but again using nontrivial \( \omega \).

Suppose that \(G\) has a \(6\)-coloring, where
\(p_1,p_2,p_3,p_4,p_5,p_6\) have colors \(A,B,C,D,E,F\), respectively. By the
definition of \(L_i\), the color of \(p_i\), for \(i\geq7\), must be
the color of some \(p_j\) with \(j\in L_i\). In particular, \(p_7\)
has color \(A\) or \(F\).

If \(p_7\) has color \(A\), the edges
\(p_7p_8,p_8p_9,p_9p_{10},p_{10}p_{11}\) successively force
\(p_8,p_9,p_{10},p_{11}\) to have colors \(B,F,C,A\), contradicting
the edge \(p_{11}p_7\).

If \(p_7\) has color \(F\), the edges
\(p_7p_{12},p_{12}p_{13},p_{13}p_{14},p_{14}p_{15}\) successively
force \(p_{12},p_{13},p_{14},p_{15}\) to have colors \(B,A,E,F\),
contradicting the edge \(p_{15}p_7\).

Hence \(\chi(G)\geq7\). Equation~\eqref{eq:incompatibility-bound}
now gives
\[
 7\leq\chi(G)\leq\dim W(E_6).
\]
\end{proof}

\begingroup
\small
\bibliographystyle{plain}
\bibliography{references}
\endgroup

\end{document}